\documentclass[a4paper]{amsart}

\usepackage{latexsym, amssymb,amsmath, amsthm,amsfonts}

\newcommand{\kk}{\Bbbk}
\newcommand{\kv}{{\kk[V]}}

\newcommand{\kvg}{{\kk[V]^{G}}}

\def\sep{\operatorname{sep}}
\newcommand{\bsep}{\beta_{\sep}}
\newcommand{\NGV}{{\mathcal{N}_{G,V}}}

\def\SL{\operatorname{SL}}

\def\reg{\operatorname{reg}}
\def\SL2{\operatorname{SL}_{2}(K)}

\def\GL2{\operatorname{GL}_{2}(K)}
\def\Ga{{\mathbb G}_{a}}
\def\Gm{{\mathbb G}_{m}}
\def\INVSL2{$K[V]^{operatorname{SL}_{2}(K)}$}
\def\INVSO2{$K[V]^{operatorname{SO}_{2}(K)}$}
\def\INVGL2{$K[V]^{operatorname{GL}_{2}(K)}$}

\def\Magma{{\sc Magma }}

\def\GL{\operatorname{GL}}
\def\SL{\operatorname{SL}}
\def\id{\operatorname{id}}

\def\Stab{\operatorname{Stab}}

\def\Ind{\operatorname{Ind}}

\def\Z{\mathbb{Z}}
\def\N{\mathbb{N}}
\def\Aut{\operatorname{Aut}}

\def\ord{\operatorname{ord}}

\newtheorem{Lemma}{Lemma}[section]
\newtheorem{Theorem}[Lemma]{Theorem}
\newtheorem{Corollary}[Lemma]{Corollary}

\newtheorem{Prop}[Lemma]{Proposition}
\newtheorem{Conj}[Lemma]{Conjecture}

\theoremstyle{definition}

\theoremstyle{remark}
  \newtheorem{rem}[Lemma]{Remark}
\newtheorem{eg}[Lemma]{Example}

\newtheoremstyle{Acknowledgments}
  {}
    {}
     {}
     {}
    {\bfseries}
    {}
     {.5em}
     {\thmname{#1}\thmnumber{ }\thmnote{ (#3)}}
\theoremstyle{Acknowledgments}
\newtheorem{ack}{Acknowledgments.}

\title[Zero-separating invariants]
{Zero-separating invariants for finite groups}

\author{Jonathan Elmer}
\address{University of Aberdeen\\
King's College, Aberdeen\\
AB24 3UE}
\email{j.elmer@abdn.ac.uk}

\author{Martin Kohls}
\address{Technische Universit\"at M\"unchen \\
 Zentrum Mathematik-M11\\
Boltzmannstrasse 3\\
 85748 Garching, Germany}
\email{kohls@ma.tum.de}

\date{\today}

\begin{document}
\maketitle

\begin{abstract}
We fix a field $\kk$ of characteristic $p$.
For a finite group $G$ denote by $\delta(G)$ and $\sigma(G)$ respectively the minimal number $d$, such that
for any finite dimensional representation $V$ of $G$ over $\kk$ and any $v\in
V^{G}\setminus\{0\}$ or $v\in V\setminus\{0\}$ respectively, there exists a homogeneous
invariant  $f\in\kk[V]^{G}$ of positive degree at most $d$ such that $f(v)\ne
0$. Let $P$ be a Sylow-$p$-subgroup of
$G$ (which we take to be trivial if the group order is not divisble by $p$). We show that $\delta(G)=|P|$. If $N_{G}(P)/P$ is cyclic, we show
$\sigma(G)\ge|N_{G}(P)|$. If $G$ is $p$-nilpotent and $P$ is not normal in $G$, we
show $\sigma(G)\le \frac{|G|}{l}$, where $l$ is the smallest prime divisor of
$|G|$. These results extend known results in the non-modular case to the
modular case.
\end{abstract}

\section{Introduction}\label{SecIntro}

Let $G$ be a linear algebraic group over an algebraically closed field $\kk$,
$V$ a finite dimensional rational representation of $G$ (which we will call a $G$-module), and denote by $\kv$ the ring of polynomial
functions $V \rightarrow \kk$. The action of $G$ on $V$ induces an action of
$G$ on $\kv$ via $g\cdot f (v) := f(g^{-1}v)$ for $g\in G$,  $f\in\kv$ and
$v\in V$. The set of $G$-invariant polynomial functions under this action is
denoted by $\kvg$, and inherits a natural grading from $\kv$, since the given action is degree-preserving. We denote by
$\kv^{G}_{d}$ the set of polynomial invariants of degree $d$ and the zero-polynomial, and by
$\kv^{G}_{\leq d}$ the set of polynomial invariants of degree at most
$d$. For any subset $S$ of $\kv$, we define $S_{+}$ as those elements of $S$
with constant term zero.

A linear algebraic group $G$ is said to be \emph{reductive} if for every
$G$-module $V$ we have that, for all nonzero $v \in V^G$, there exists  $f \in \kk[V]_{+}^G$ such that $f(v) \neq 0$. It is said to be \emph{linearly reductive} if
for all nonzero  $v \in V^G$ there exists $f \in \kk[V]^G_1$ such that $f(v) \neq 0.$ Obviously linear reductivity implies reductivity. Denote by $\mathcal{N}_{G,V}$ the nullcone of $V$, that is
\[\mathcal{N}_{G,V} = \left\{v \in V| \quad f(v)=0  \quad \text{ for all } f
  \in \kk[V]_{+}^G \right\}.\] 
Note that the nullcone is the vanishing set of the ``Hilbert Ideal''  $I_{G,V}$ of
$\kv$, which is the ideal of $\kv$ generated by all homogeneous invariants of positive degree.
Then  $G$ is reductive if for any $G$-module $V$,  one has that $\mathcal{N}_{G,V} \cap V^G = \{0\}$.

The concept of reductivity is important in both invariant theory and the
theory of linear algebraic groups. One of the most celebrated results of 20th
century invariant theory is the theorem of Nagata \cite{NagataHilbert14} and
Popov \cite{Popov} which states that $\kk[X]^G$ is finitely generated for all
affine $G$-varieties $X$ if and only if $G$ is reductive.

The  first part of the present article is motivated by a simple, perhaps even facetious,
question: are there any ``quadratically reductive'' groups? The reader can
probably guess the definition, but we explain this in detail, while
introducing some useful terminology. Let $G$ be a linear algebraic group over
$\kk$ and $V$ a $G$-module. We shall say a subset $S \subseteq \kvg$ is a
$\delta$-\emph{set} if, for all $v \in V^G \setminus \mathcal{N}_{G,V}$, there
exists an $f\in S_{+}$ such that $f(v) \neq 0$. We shall call a subalgebra of $\kk[V]^G$ a $\delta$-\emph{subalgebra} if it is a $\delta$-set. The quantity $\delta(G,V)$ is then defined as
\[
\delta(G,V) = \min\{d\ge 0|\quad \kk[V]^G_{\leq d} \, \text{ is a
  $\delta$-set} \,  \}.
\]
We will justify below that $\delta(G,V)$ is always a finite number. Finally, we define $$\delta(G):= \sup\{\delta(G,V) |\quad V \,\text{ a
   $G$-module}\},$$ where we take the supremum of an unbounded set to be infinity.

Note that if $G$ is reductive over $\kk$, the definitions above simplify: $S
\subseteq \kk[V]^G$ is then a $\delta$-set if and only if for all nonzero $v \in
V^G$, there  is an $f \in S_{+}$ with $f(v) \neq 0$. Note further, that a reductive group $G$ is linearly reductive if and only if $\delta(G)=1$. A quadratically reductive group, then, ought be a reductive group $G$ for which $\delta(G) = 2$. There are plenty of examples. We show in Section \ref{SecDelta}:

\begin{Theorem}\label{thmdeltag}
Let $G$ be a finite group, $\kk$ an algebraically closed field of characteristic $p$, and $P$ a Sylow-$p$-subgroup of $G$. Then $\delta(G) = |P|$.
\end{Theorem}

It is well known that a finite group $G$ is linearly reductive over a field $\kk$ if and only if the order of $G$ is not divisible by the characteristic of $\kk$. The theorem above can be viewed as a generalisation of this result, where we take the Sylow-$p$-subgroup to be trivial in the non-modular case.

In addition to $\delta(G)$, we also study the closely related quantity
$\sigma(G)$. The definition is as follows.  We shall say a subset $S \subseteq
\kk[V]^G$ is a $\sigma$-\emph{set} if, for all $v \in V \setminus
\mathcal{N}_{G,V}$, there exists an $f\in S_{+}$  such that $f(v) \neq 0$. 

We shall call a subalgebra of $\kk[V]^G$ a $\sigma$-\emph{subalgebra} if it is a $\sigma$-set. The quantity $\sigma(G,V)$ is then defined as
\[
\sigma(G,V) = \min\{d\ge 0|\quad \kk[V]^G_{\leq d} \,\, \text{ is a $\sigma$-set} \,
\}.
\]
 It is clear that a generating set of  the Hilbert ideal $I_{G,V}$ which
 consists of invariants is a $\sigma$-set. Therefore, since $\kv$ is Noetherian, $\kk[V]^G$ always contains a finite $\sigma$-set and the number $\sigma(G,V)$ is finite.
 Finally, we define 
\[
\sigma(G):= \sup\{\sigma(G,V) |\quad  V \text{ a $G$
   -module}\},
\] 
which can be finite or infinite. It is immediately clear that $\delta(G,V) \leq
\sigma(G,V)$ for all $G$-modules $V$, and that $\delta(G) \leq
\sigma(G)$. It is also well known that $\sigma(G)\le |G|$, e.g. from Dade's
Algorithm \cite[Proposition 3.3.2]{DerksenKemper}.

Note that $\sigma(G,V)$ can be interpreted in a few different ways. For
instance, we see that $\sigma(G,V)$ is the minimal degree $d$ such that there
exists a finite set of invariants of degree at most $d$ whose common zero set
is  $\mathcal{N}_{G,V}$. If $G$ is reductive, then a graded subalgebra $S \subseteq \kk[V]^G$ is a
$\sigma$-subalgebra if and only if $\kk[V]^G$ is a finitely generated
$S$-module (see \cite[Lemma~2.4.5]{DerksenKemper}). So in the case of
reductive groups, $\sigma(G,V)$  is the minimal degree $d$ such that
there exists a set $T$ of homogeneous invariants of degree at most $d$ such
that $\kvg$ is a finitely generated $\kk[T]$-module.
Recall that for reductive groups, $\mathcal{N}_{G,V}$ consists of those $v \in V$ such that $0 \in \overline{G\cdot v}$, where the bar denotes closure in the Zariski topology (see e.g. \cite[Lemma~2.4.2]{DerksenKemper}). In particular when $G$ is finite we have that $\mathcal{N}_{G,V}$ = \{0\}, so $\sigma(G,V)$
may be defined as the minimal degree $d$ such that there exists a finite set
of invariants of degree at most $d$ whose common zero set is $\{0\}$. 

For linearly reductive groups in
characteristic $0$, the  $\sigma$-number plays an
important role in giving upper bounds for the classical Noether number
$\beta(G,V)=\beta(\kvg)$, which is defined as the minimum degree $d$ such
that $\kk[V]^G_{\leq d}$ generates $\kk[V]^G$ as an algebra. Again, the ``global''
value $\beta(G)$ is defined as the supremum of all $\beta(G,V)$. For example, Derksen \cite[Theorem 1.1]{DerksenPolynomial} gives the
upper bound 
\[
\beta(G,V)\le \max\left\{2\,\, ,\,\,\frac{3}{8}\cdot \dim (\kvg)\cdot \sigma(G,V)^{2}\right\}.
\]
Cziszter and Domokos  \cite{DomokosCziszter} study $\sigma(G)$ for finite
groups over fields of characteristic not dividing $|G|$. In particular, they show

\begin{Prop}[{Cziszter and Domokos \cite{DomokosCziszter}}]\label{propdcnonmodularsigma}
Let $G$ be a finite group, and let $\kk$ be an algebraically closed field of characteristic not dividing $|G|$. Then $\sigma(G)=|G|$ if and only
if $G$ is cyclic. More precisely, if $G$ is not cyclic, then $\sigma(G) \leq |G|/l$ where $l$ is the smallest prime dividing $|G|$.
\end{Prop}

\begin{proof} \cite[Theorem~7.1]{DomokosCziszter} states that if $G$ is not cyclic, then $\sigma(G) \leq |G|/l$ where $l$ is the smallest prime dividing $|G|$. In particular $\sigma(G)<|G|$ when $G$ is not cyclic. Conversely, \cite[Corollary~5.3]{DomokosCziszter} states that if $G$ is abelian, $\sigma(G) = \exp(G)$. In particular, if $G$ is cyclic, $\sigma(G)=|G|.$
\end{proof}

In sections \ref{SecSigmaGeneral} and \ref{SecSigmaFinite} we generalise some
results of Cziszter and Domokos  to fields of
arbitrary characteristic. In particular, we prove the following version of the
above for the modular case (where $N_{G}(P)$ is the normalizer of the subgroup
$P$ in $G$):

\begin{Theorem}\label{thmsigmaleg} 
Suppose $G$ is a finite group, and that $\kk$ is an algebraically closed field of characteristic $p$, where $p$ divides $|G|$. Let $P$ be
  a Sylow-$p$-subgroup of $G$. Also let $l$ denote the smallest prime divisor
  of $|G|$. Then the following holds:
\begin{enumerate}
\item[(a)] If $\sigma(G)=|G|$, then $N_{G}(P)/P$ is a cyclic group. If
  additionally  $P$ is abelian and $G\ne P$, then  $N_{G}(P)/P$ is also non-trivial.
\item[(b)] If $N_G(P)/P$ is cyclic, then $\sigma(G) \geq |N_G(P)|$. In
  particular $\sigma(G)=|G|$ when $P$ is normal in $G$ and $G/P$ is cyclic.
\item[(c)] If $G$ is $p$-nilpotent and $P$ is not normal, then $\sigma(G)\le\frac{|G|}{l}$.
\end{enumerate}
\end{Theorem}

Another quantity associated with $\delta(G,V)$ and $\sigma(G,V)$, which has
attracted some attention in recent years, is $\beta_{\sep}(G,V)$.  It is
defined as follows: a subset $S \subseteq \kk[V]^G$ is called a
\emph{separating set} if, for any pair $v, w \in V$ such that there
exists $f \in \kk[V]^G$ with $f(v) \neq f(w)$, there exists $s \in S$ with
$s(v) \neq s(w)$. Then $\beta_{\sep}(G,V)$ is defined as 
\[
\beta_{\sep}(G,V) = \min\{d\ge 0|\quad \kk[V]^G_{\leq d} \, \text{ is a separating
  set} \, \},
\]
and once more, the ``global'' value $\beta_{\sep}(G)$ is defined to be the
supremum over all $\beta_{\sep}(G,V)$.

Our point of view is that $\delta$- and $\sigma$-sets are ``zero-separating''
sets. This leads to the following inequalities:

\begin{Prop}\label{DeltaLeSigmaLeBeta} 
Let $G$ be a linear algebraic group and $V$ a $G$-module. Then
\[\delta(G,V) \leq \sigma(G,V) \leq \beta_{\sep}(G,V) \leq \beta(G,V).\]
\end{Prop}

\begin{proof} 
The first and last inequalities are trivial. Assume $S\subseteq \kk[V]^{G}_{+}$ is a separating set. It is enough to show that $S$ cuts out the
Nullcone. Indeed, if $v\in V\setminus \NGV$, then there is an
$f\in\kk[V]^{G}_{+}$ such that $f(v)\ne 0 = f(0)$. Thus there is an $s\in S$
such that $s(v)\ne s(0)=0$. Consequently, if $S \subseteq \kk[V]_{+}^G$ is a separating set then it is a $\sigma$-set, and we get the second inequality. 
\end{proof}

The above implies that one has, for any linear algebraic group
$G$, $$
\delta(G) \leq \sigma(G) \leq \beta_{\sep}(G) \leq \beta(G).
$$  
For finite groups, Derksen and Kemper \cite[Theorem~3.9.13]{DerksenKemper} showed that $\bsep(G) \leq |G|$, independently of the characteristic of $\kk$. For this reason we obtain as a consequence of Theorem 1.1, for a finite group $G$ with Sylow-$p$-subgroup $P$,
\[
|P| = \delta(G) \leq \sigma(G) \leq \bsep(G) \leq |G|.
\]
Fleischmann \cite{FleischmannNoetherBound} and Fogarty \cite{FogartyNoetherBound} proved independently that if $p$ does not divide the order of $G$, then we have the stronger result that $\beta(G) \leq |G|$ (the result in characteristic zero is due to Emmy Noether, hence the name). In that case we obtain
\[
1=\delta(G) \leq \sigma(G) \leq \bsep(G) \leq \beta(G) \leq |G|.
\]
In this paper we focus mainly on the case where $G$ is a finite group. However, a
subsequent paper dealing with infinite algebraic groups is in preparation. As
for some of the results in the present paper the proofs for infinite groups
are not more difficult than those for finite groups, we will give the proofs for the most
general case.

\section{The $\delta$-number for finite groups}\label{SecDelta}

The goal of this section is to prove Theorem \ref{thmdeltag}, which we do in a
series of basic propositions. 

\begin{Prop}\label{propuleqv} 
Let $G$ be a reductive group and let $U$ be a $G$-submodule of $V$. Then $\delta(G,U) \leq \delta(G,V)$.
\end{Prop}

\begin{proof}  
Let $d = \delta(G,V)$ and take $u \in U^{G} \setminus \mathcal{N}_{G,U}
$.  Clearly $u \in V^{G}\setminus\{0\}$, and reductivity implies $u \not\in  \mathcal{N}_{G,V}$. It follows that there exists an
  $f \in \kk[V]_{+,\leq d}^G$  with $f(u) \neq 0$. Now set $g:=
  f|_U$. Then we have $g \in \kk[U]_{+,\leq d}^G$ and $g(u) \neq 0$. This shows that $\delta(G,U) \leq d$. 
\end{proof}

\begin{Prop}\label{propdctsum}
Let $G$ be a reductive group, $V_1, V_2$ be $G$ modules and $W=V_1 \oplus V_2$. Then $\delta(G,W) = \max\{\delta(G,V_1),\delta(G,V_2)\}$.
\end{Prop}

\begin{proof} 
We have $d:=\max\{\delta(G,V_1),\delta(G,V_2)\} \leq \delta(G,W)$ by the previous
proposition.  Take $w=v_1+v_2 \in W^G\setminus\mathcal{N}_{G,W}$ with
$v_i \in V_i$ for $i=1,2$. Then $v_i \in V_i^G$ for $i=1,2$. Clearly $w\ne 0$, hence $v_{1}\ne 0$ or $v_{2}\ne 0$.
Without loss of generality assume $v_1 \neq 0$. Reductivity implies $v_{1}\in
V_{1}^{G}\setminus\mathcal{N}_{G,V_{1}}$. Hence there exists an $f \in
\kk[V_1]^G_{+,\leq \delta(G,V_{1})}$ with $f(v_1) \neq 0$. As we have the
$G$-algebra 
inclusion $\kk[V_{1}]\subseteq \kk[V_{1}\oplus V_{2}]$, $f$
can be viewed as an element of $\kk[W]^{G}_{+,\le d}$  satisfying $0\ne f(v_{1})=f(v_{1}+v_{2})=f(w)$.  This shows that $\delta(G,W) \leq d$ as required.
\end{proof}

\begin{rem}\label{remdctsum} 
Using the above and induction, it follows that $$\delta(G,W) =
\max\{\delta(G,V_i)|\quad  i = 1, \ldots , n\}$$ whenever $W = \bigoplus_{i=1}^n V_i$ is a finite direct sum of $G$-modules.
\end{rem}

\begin{Prop}\label{propvreg}
Let $G$ be a finite group. Then $\delta(G) = \delta(G,V_{\reg})$ where
$V_{\reg}:=\kk G$ denotes the regular representation of $G$ over $\kk$.
\end{Prop}

\begin{proof} 
It is well-known that, given any  $G$-module $V$, we have an
embedding $V \hookrightarrow V^{n}_{\reg}$ for $n=\dim_{\kk}(V)$ (choosing an arbitrary basis of $V^{*}$ yields an
epimorphism $(\kk G)^{n}\twoheadrightarrow V^{*}$, and dualizing yields
the desired embedding as $V_{\reg}=\kk G$ is self dual - see also \cite[proof of
Corollary 3.11]{DraismaSeparating}).
Now by Proposition \ref{propuleqv} and Remark \ref{remdctsum} we obtain
$$\delta(G,V) \leq \delta(G,V^{n}_{\reg}) = \delta(G,V_{\reg}).$$
The result now follows from the definition of $\delta(G)$. 
\end{proof}

The proof of the following Proposition, which is key to proving Theorem \ref{thmdeltag}, is similar to \cite[Proposition
8]{KohlsKraft}, but our point of view is different and we get a new
result. Also note that if $G$ is a $p$-group, Theorem \ref{thmdeltag} and
Propositon \ref{DeltaLeSigmaLeBeta} imply $|G| = \delta(G) = \sigma(G) =
  \bsep(G)$, strengthening the result in \cite[Proposition~8]{KohlsKraft}.

\begin{Prop}\label{propdeltavreg} 
Let $G$ be a finite group, $\kk$ a field of characteristic $p$, and let $P$ be a Sylow-$p$-subgroup of $G$ (if $p=0$ or does not divide the order of $G$, take $P$ to be the trivial group). Then $\delta(G,V_{\reg}) = |P|$.
\end{Prop}

\begin{proof} Let $\{v_g|\,\, g \in G\}$ be a $\kk$-basis for
  $V:=V_{\reg}$. The fixed point space $V^G$ of $V$ is 1-dimensional and spanned
  by $v:=\sum_{g \in G}v_g$. Write $\kk[V] = \kk[x_g:\,\, g \in G]$ where
  $\{x_g|\,\, g \in G\}$ is the basis of $V^*$ dual to $\{v_g|\,\, g \in
  G\}$. Since $V$ is a permutation representation, the ring of invariants
  $\kk[V]^G$ is generated as a vector space by \emph{orbit sums} of monomials, that is, by invariants of the form
$$o_G(m):=\sum_{m' \in G\cdot m} m'$$ where $m := \prod_{g \in G}x_g^{n_g}$ is
a monomial in $\kk[x_g: g \in G]$ and $G \cdot m$ denotes the orbit of
$m$. Clearly then for any $g \in G$ we have $x_g(v) = 1$, and therefore for any monomial
$m\in\kv$ we have $m(v)=1$.
It follows that for any monomial $m$, we have
$$o_G(m)(v) = \sum_{m' \in G\cdot m}m'(v) = \sum_{m' \in G\cdot m} 1 =  |G\cdot m|.$$

Now let $0 \neq u \in V^G$. Then $u = \lambda v$ for some nonzero $\lambda \in \kk$. Set $m:=\prod_{g \in P}x_g$ and $f:= o_G(m)$. Note that $f$ is an invariant of degree $|P|$, and that
\[f(u) = \lambda^{|P|} |G \cdot m| =\lambda^{|P|} [G:\Stab_{G}(m)]=\lambda^{|P|}\frac{|G|}{|P|} \neq 0 \in \kk.\]
This shows that $\delta(G,V) \leq |P|$. 

Conversely, any $f \in \kk[V]^G$  can be written as a $\kk$-sum of orbit
sums of monomials. Therefore, if $f(v) \neq 0$ for some homogeneous invariant $f$, for some monomial $m$ of the same degree as
$f$ we must have $o_G(m)(v) \neq 0$. This means that $|G \cdot
m|=[G:\Stab_{G}(m)]$ is not divisible by $p$. Hence, $|P|$ divides
$|\Stab_G(m)|$. Therefore $\Stab_G(m)$ contains a Sylow-$p$-subgroup $Q$ of
$G$. Consequently, if $m$ is divisible by some $x^k_g$, $m$ must also be divisible by $x^k_{qg}$ for all $q \in Q$. In particular, $\deg(m) \geq |Q| = |P|$. This shows that $\deg(f) \geq |P|$, and hence $\delta(G,V) \geq |P|$.
\end{proof}

\begin{proof}[Proof of Theorem \ref{thmdeltag}]
Combine Propositions \ref{propvreg} and \ref{propdeltavreg}.
\end{proof}

\section{Relative results for the $\sigma$-number}\label{SecSigmaGeneral}

In this section we prove mainly relative results about $\sigma(G)$ for both finite and
infinite groups $G$. Many of these are extensions of results in
\cite{DomokosCziszter} to fields of arbitrary characteristic and to infinite groups. 

\begin{Prop}\label{propsigmauleqv} Let $G$ be a reductive group and let $U$ be
  a $G$-submodule of $V$. Then $\sigma(G,U) \leq \sigma(G,V)$. 
\end{Prop}

\begin{proof} 
Let $d:=\sigma(G,V)$ and take $u\in U\setminus \mathcal{N}_{G,U}$. This
implies $0\not\in\overline{G\cdot u}$. As $U$ is a
closed subset of $V$, it does not matter if the closure of $G\cdot u$ is taken
in $U$ or in $V$. Now the reductivity of $G$ implies $u\not\in \NGV$, and
therefore there exists an $f\in\kk[V]^{G}_{+,\le d}$ with $f(u)\ne 0$. Then
$f|_{U}\in \kk[U]^{G}_{+,\le d}$ also separates $u$ from $0$, hence
$\sigma(G,U)\le d$.
\end{proof}

Note that for non-reductive groups, it is not always the case that $U
\subseteq V$ implies $U \setminus \mathcal{N}_{G,U} \subseteq V \setminus \mathcal{N}_{G,V}$.
For example take the action of the additive group $\Ga=(\kk,+)$ on $V=\kk^{2}$ via $t*(a,b):=(a+tb,b)$ for
$t\in\Ga$ and $(a,b)\in V$. We write $\kk[V]=\kk[x,y]$.  Take the point $u=(1,0)$ in the submodule
$U:=\kk\cdot(1,0)$. As the action of $\Ga$ on $U$ is trivial,
$\kk[U]^{\Ga}=\kk[x|_{U}]$, so we have $u\in U\setminus
\mathcal{N}_{\Ga,U}$. But $\kk[V]^{\Ga}=\kk[y]$, so $u\in\mathcal{N}_{\Ga,V}$.

For arbitrary (even non-reductive) algebraic groups, we have the following
result:

\begin{Lemma}\label{sigmDirectSummand}
Let $G$ be an arbitrary group and let $U$ and $V$ be $G$-modules such that $U$
is a direct summand of $V$. Then $\sigma(G,U) \leq \sigma(G,V)$. 
\end{Lemma}

\begin{proof}
Take a $u\in U\setminus\mathcal{N}_{G,U}$ and an $f\in \kk[U]^{G}_{+}$ such
that $f(u)\ne 0$. As $U$ is a direct summand of $V$, we have an
inclusion of $G$-algebras $\kk[U]\subseteq \kk[V]$, hence we can view $f$ as an
element of $\kk[V]^{G}_{+}$. As $f(u)\ne 0$, we have $u\in
V\setminus\NGV$. Therefore there is a $g\in \kk[V]^{G}_{+,\le\sigma(G,V)}$ such
that $g(u)\ne 0$. Then $g|_{U}\in \kk[U]^{G}_{+,\le\sigma(G,V)}$ satisfies
$g|_{U}(u)\ne 0$, hence $\sigma(G,U)\le\sigma(G,V)$.
\end{proof}

The following basic result also appears in  Cziszter and Domokos \cite[Lemma~5.1]{DomokosCziszter}, but we give a simpler argument here:

\begin{Prop}\label{SigmaDirectSumFiniteGroups}
 Let $G$ be a finite group and suppose $W=V_1\oplus V_2$, where $V_1$, $V_2$ and $W$ are $G$-modules. Then $\sigma(G,W) = \max\{\sigma(G,V_1),\sigma(G,V_2)\}$.
\end{Prop}

\begin{proof}
We have $d:=\max\{\sigma(G,V_1),\sigma(G,V_2)\}\le \sigma(G,W)$ by Proposition
\ref{propsigmauleqv}. Conversely, take a nonzero $w=v_{1}+v_{2}\in W$ with
$v_{i}\in V_{i}$ for $i=1,2$. Without loss we can assume $v_{1}\ne 0$. Then
there is an $f\in\kk[V_{1}]^{G}_{+,\le \sigma(G,V_{1})}$ with $f(v_{1})\ne 0$. As
in the proof of Proposition \ref{propdctsum}, we can view $f$ as an element of
$\kk[W]^{G}_{+,\le d}$ such
that $f(w)=f(v_{1}+v_{2})=f(v_{1})\ne 0$. Therefore, $\sigma(G,W)\le d$.
\end{proof}

Note that the above is not true for reductive algebraic groups in general; a
counterexample is provided in \cite[Remark~5.2]{DomokosCziszter}. However,
even for infinite groups the $\sigma$-value of vector invariants has an interesting stabilization
property, which was observed by Domokos
\cite[Remark 3.3]{DomokosSep}. As Domokos only remarks that the proof of the
following proposition can be given with the same methods as in his paper
\cite{DomokosSep} (where a similar result for $\beta_{\sep}$ is given), we
give the proof here for the sake of completeness.

\begin{Prop}[Domokos]\label{Domokos}
Assume $G$ is an arbitrary (possibly infinite) group acting linearly on an
$n$-dimensional vector space $V$ (the action need not even be rational). Then
\[
\sigma(G,V^{m})=\sigma(G,V^{n})\quad\quad\text{ for all }\,\,m\ge n.
\]
\end{Prop}

Note that for finite groups, by Proposition \ref{SigmaDirectSumFiniteGroups}
we have more precisely $\sigma(G,V^{m})=\sigma(G,V)$ for all $m$.

Under the hypotheses of the theorem, we first show the following:

\begin{Lemma}\label{DomoksLemma}
Let $v=(v_{1},\ldots,v_{m})$ and $u=(u_{1},\ldots,u_{m})\in V^{m}$ be such that
their components span the same $\kk$-vector subspace of $V$, i.e.
\[
\langle v_{1},\ldots, v_{m}\rangle_{\kk}=\langle u_{1},\ldots, u_{m}\rangle_{\kk}.
\]
Then we have
\[
v\in\mathcal{N}_{G,V^{m}}\Leftrightarrow u\in\mathcal{N}_{G,V^{m}}.
\]
\end{Lemma}

\begin{proof}
Assume $v\not\in\mathcal{N}_{G,V^{m}}$. Then there is an $f\in \kk[V^{m}]^{G}_{+}$
with $f(v)\ne 0$. By assumption, we can find $\alpha_{ij}\in \kk$ such that
\[
v_{i}=\sum_{j=1}^{m}\alpha_{ij}u_{j}\quad \text{ for all }i=1,\ldots,m.
\]
Write $f=f(x_{1},\ldots,x_{m})\in\kk[V^{m}]^{G}$, where each $x_{j}$ belongs
to a set of coordinates of an element of $V$, and set
\[
h(x_{1},\ldots,x_{m}):=f\left(\sum_{j=1}^{m}\alpha_{1,j}x_{j},\sum_{j=1}^{m}\alpha_{2,j}x_{j},\ldots,\sum_{j=1}^{m}\alpha_{m,j}x_{j}\right)\in\kk[V^{m}]_{+}.
\]
It is immediately checked that $h$ inherits $G$-invariance from $f$, so
$h\in\kk[V^{m}]^{G}_{+}$. Now
\begin{eqnarray*}
h(u)=h(u_{1},\ldots,u_{m})&=&f\left(\sum_{j=1}^{m}\alpha_{1,j}u_j,\sum_{j=1}^{m}\alpha_{2,j}u_j,\ldots,\sum_{j=1}^{m}\alpha_{m,j}u_j\right)\\
&=&f(v_{1},\ldots,v_{m})=f(v)\ne 0,
\end{eqnarray*}
hence $u\not\in \mathcal{N}_{G,V^{m}}$. We have shown: If $\langle
v_{1},\ldots, v_{m}\rangle_{\kk}\subseteq\langle u_{1},\ldots,
u_{m}\rangle_{\kk}$, then we have the implication: 
$v\not\in\mathcal{N}_{G,V^{m}}\Rightarrow u\not\in \mathcal{N}_{G,V^{m}}$. The
reverse implication follows in the same way, so we are done.
\end{proof}

\begin{proof}[Proof of Proposition ~\ref{Domokos}]
By Lemma \ref{sigmDirectSummand} we have $\sigma(G,V^{n})\le\sigma(G,V^{m})$,
so we have to show the reverse inequality.
Take a  point $v=(v_{1},\ldots,v_{m})\in
V^{m}\setminus\mathcal{N}_{G,V^{m}}$. As  $\dim V=n$, we can find a point
$u=(u_{1},\ldots,u_{n},0,\ldots,0)\in V^{m}$ such that $\langle
u_{1},\ldots,u_{n}\rangle=\langle v_{1},\ldots,v_{m}\rangle$. By Lemma
\ref{DomoksLemma} we have $u\not\in\mathcal{N}_{G,V^{m}}$. Hence there is an
$f\in \kk[V^{m}]^{G}_{+}$ such that $f(u)\ne 0$. Then $f|_{V^{n}}\in \kk[V^{n}]^{G}_{+}$ satisfies
$f(\tilde{u})\ne 0$, where $\tilde{u}=(u_{1},\ldots,u_{n})\in
V^{n}$. Therefore, $\tilde{u}\not\in\mathcal{N}_{G,V^{n}}$, so there is an
$\tilde{f}\in \kk[V^{n}]^{G}_{+,\le \sigma(G,V^{n})}$ such that
$\tilde{f}(\tilde{u})\ne 0$. As we have a $G$-algebra inclusion
$\kk[V^{n}]\subseteq \kk[V^{m}]$, we can take $\tilde{f}$ as an element of
$\kk[V^{m}]^{G}_{+,\le\sigma(G,V^{n})}$ satisfying
$\tilde{f}(u)=\tilde{f}(\tilde{u})\ne 0$. As in the proof of Lemma
\ref{DomoksLemma}, there is an $h\in \kk[V^{m}]^{G}_{+,\le\sigma(G,V^{n})}$
satisfying $h(v)=\tilde{f}(u)\ne 0$. This shows $\sigma(G,V^{m})\le\sigma(G,V^{n})$.
\end{proof}

Now we restrict again to finite groups and give two corollaries of
Propositions~\ref{propsigmauleqv} and~\ref{SigmaDirectSumFiniteGroups}.

\begin{Corollary}\label{corsigmavreg} Let $G$ be a finite group. Then $\sigma(G)=\sigma(G,V_{\reg})$ where $V_{\reg}$ denotes the regular representation of $G$ over $\kk$.
\end{Corollary}

\begin{proof}
As the proof of Proposition \ref{propvreg}.
\end{proof}

Recall that the decomposition of the regular representation into indecomposables gives the complete list
of projective indecomposable modules. As a consequence of this,
Proposition \ref{SigmaDirectSumFiniteGroups} and the above corollary, we have

\begin{Corollary}\label{corprojective} Let $G$ be a finite group. Then $$\sigma(G) = \max\{\sigma(G,U)\mid U \ \text{is a projective indecomposable $G$-module}\}.$$
\end{Corollary}

\begin{Prop}\label{propsigmanormsubgroup} Let $G$ be a group and let $N$ be a normal subgroup of $G$ with finite index. Let $V$ be a $G$-module. Then
\[\sigma(G,V) \leq \sigma(N,V)\sigma(G/N)\le \sigma(N)\sigma(G/N),\]
so particularly we have $\sigma(G)\le\sigma(N)\sigma(G/N)$.
\end{Prop}

\begin{proof} 
Only the first inequality needs to be shown. Choose a  finite $\sigma$-subset
$\{f_1,f_2,\ldots, f_n\}$ of $\kk[V]_{+}^N$, with $\deg(f_i) \leq \sigma(N,V)$
for all $i= 1,\ldots, n$. Take a left-transversal $\{g_1,g_2,\ldots,g_r\}$ of $N$ in $G$, that is to
say, $G=\bigcup_{i=1}^r g_iN$ where $r=[G:N]$. Let $v \in V \setminus
\mathcal{N}_{G,V}$. As a $G$-invariant separating $v$ from zero is clearly
also an $N$-invariant, we see that $v \in V \setminus \mathcal{N}_{N,V}$. Consequently, the vector
\[(f_1(v),f_2(v),\ldots,f_n(v)) \in \kk^n\]
is not zero, and nor is the vector
\[\hat{v} := (g_{1}(f_1)(v),g_{2}(f_1)(v),\ldots,g_{r}(f_{1})(v),\ldots,g_{1}(f_{n})(v),\ldots,g_{r}(f_{n})(v))\in \kk^{nr}.\]
We may define an action on $\kk^{nr}$ so that it becomes isomorphic to $n$
copies of the regular representation of $G/N$, i.e. to $V_{\reg,G/N}^{n}$ in such a
way that the action of $G/N$ on $\hat{v}$ is given by
\begin{equation}\label{ActionOnknr}
g^{-1}N\cdot \hat{v}= ((gg_{1})(f_1)(v),\ldots,(gg_{r})(f_{1})(v),\ldots,(gg_{1})(f_{n})(v),\ldots,(gg_{r})(f_{n})(v))
\end{equation}
for all $g\in G$. Since $G/N$ is finite, its nullcone is zero, and as
$\hat{v}\ne 0$ we can find an invariant $h \in \kk[V_{\reg,G/N}^{n}]_{+,\leq \sigma(G/N)}^{G/N}$ such that $h(\hat{v}) \neq 0$. 
Now consider the
polynomial $$\hat{h}:=h(g_1(f_1),g_2(f_1),\ldots,g_r(f_1),\ldots,
g_1(f_n),\ldots, g_r(f_n))\in \kk[V]_{+}.$$ Notice that $\hat{h}(v) = h(\hat{v}) \neq 0$, and
that $\deg(\hat{h}) \leq \sigma(N,V)\sigma(G/N)$. It remains to show that
$\hat{h}$ is $G$-invariant. From the definition of the
action of $G/N$ on $\kk^{nr}$, we see that for any $g\in G$ and $u\in V$
we have
\begin{eqnarray*}
(g\hat{h})(u)&=&h(gg_1(f_1),\ldots,gg_r(f_1),\ldots,
gg_1(f_n),\ldots, gg_r(f_n))(u)\\
&=&h(gg_1(f_1)(u),\ldots,gg_r(f_1)(u),\ldots,
gg_1(f_n)(u),\ldots, gg_r(f_n)(u))\\
&\stackrel{~\eqref{ActionOnknr}}{=}&h(g^{-1}N\cdot (g_1(f_1)(u),\ldots,g_r(f_1)(u),\ldots,
g_1(f_n)(u),\ldots, g_r(f_n)(u)))\\
&=&(gN\cdot h)( g_1(f_1)(u),\ldots,g_r(f_1)(u),\ldots,
g_1(f_n)(u),\ldots, g_r(f_n)(u))\\
&\stackrel{(*)}{=}&h( g_1(f_1)(u),\ldots,g_r(f_1)(u),\ldots,
g_1(f_n)(u),\ldots, g_r(f_n)(u))=\hat{h}(u),
\end{eqnarray*}
where in $(*)$ we used that $h$ is $G/N$ invariant. Hence,
$g(\hat{h})=\hat{h}$ for all $g\in G$, that is $\hat{h} \in \kk[V]^G$, so we are done.
\end{proof}

\begin{Prop}\label{propsigmasubgroup} 
Let $G$ be a group and let $H$ be a subgroup of $G$ with finite index. Let $V$ be a $G$-module. Then
\[\sigma(G,V) \leq \sigma(H,V)[G:H]\le\sigma(H)[G:H],\]
so in particular we have $\sigma(G)\le\sigma(H)[G:H]$.
\end{Prop}

\begin{proof} 
As in Proposition \ref{propsigmanormsubgroup}, we can find a finite
$\sigma$-subset $\{f_1,f_2,\ldots, f_n\} \subseteq \kk[V]_{+}^H$ with the
property that $\deg(f_i) \leq \sigma(H,V)$ for all $i$. Let
$\{g_1,g_2,\ldots,g_r\}$ be a left-transversal of $H$ in $G$. Take a new
independent variable $T$ on which $G$ acts trivially, and form the polynomial
ring $\kk[V][T]$. As the polynomial $\sum_{j=1}^n
f_j T^{j-1}$ is $H$-invariant, its relative ``norm''
\[z(T):=\prod_{i = 1}^r\left(g_{i}\left( \sum_{j=1}^n f_j T^{j-1}\right)\right) =\prod_{i = 1}^r\left( \sum_{j=1}^n g_i (f_j) T^{j-1}\right) \in \kk[V]_{+}[T].\]
is $G$-invariant, hence the coefficients of $z$ as a polynomial in $T$ are $G$-invariant.
Let $S\subseteq \kk[V]^{G}_{+}$ be this set of coefficients of $z$.  We claim that $S$ is a $\sigma$-set. 
Suppose that $v \in V$ is such that $f(v)=0$ for all $f \in S$. We must show that $v \in \mathcal{N}_{G,V}$. We have that $z(T)(v)$ is the zero polynomial, i.e.
\[\prod_{i = 1}^r\left( \sum_{j=1}^n f_j(g_i^{-1}v) T^{j-1}\right) = 0 \in \kk[T].\]
Since $\kk[T]$ is an integral domain, this implies that one of the factors of the above is zero, that is, for some $i \in\{ 1,\ldots, r\}$,
$$ \sum_{j=1}^n f_j(g_i^{-1}v) T^{j-1}=0 \in \kk[T].$$
This implies that $f_j(g_i^{-1}v) = 0$ for all $j=1,\ldots,n$. Since the set $\{f_1,f_2,\ldots, f_n\} \subseteq \kk[V]_{+}^H$ is a $\sigma$-set, we deduce that $f(g_i^{-1}v) = 0$ for all $f \in \kk[V]_{+}^H$. In particular, $f(g_i^{-1}v)=(g_{i}(f))(v) = 0$ for all $f \in \kk[V]_{+}^G$. This means that $f(v)=0$ for all $f \in \kk[V]_{+}^G$, i.e. that $v \in \mathcal{N}_{G,V}$ as required.
\end{proof}

The following is the first statement of Theorem \ref{thmsigmaleg} (a):
\begin{Corollary}\label{corngpmodp} 
Let $G$ be a finite group, and let $\kk$ be an algebraically closed field of characteristic $p$. Let $P$ be a Sylow-$p$-subgroup of $G$ (if $p$ does not divide $|G|$, take $P$ the trivial group) and suppose $N_G(P)/P$ is not cyclic. Then $\sigma(G)<|G|$.
\end{Corollary}

\begin{proof} By Propositions \ref{propsigmasubgroup},
  \ref{propsigmanormsubgroup} and \ref{propdcnonmodularsigma},
    we have $$\sigma(G) \leq \sigma(N_G(P))[G:N_G(P)]\leq \sigma(P)\sigma(N_G(P)/P)[G:N_G(P)]$$
$$<|P|[N_G(P):P][G:N_G(P)] = |G|.$$
\end{proof}

\begin{Lemma}\label{sigmaSubgroupsEasy}
Assume $G$ is a reductive group with a closed subgroup $H$ of finite index, and $V$ a $G$-module. Then
$\sigma(H,V)\le \sigma(G,V)$.
\end{Lemma}

\begin{proof}
Let $v\in V\setminus\mathcal{N}_{H,V}$. Clearly this implies
$0\not\in\overline{Hv}$. Let $g_{1},\ldots,g_{r}$ be a left transversal
of $H$ in $G$. Then we have
\[
\overline{G v}=\overline{\left(\bigcup_{i=1}^{r}g_{i}H\right)\cdot
  v}=\overline{\bigcup_{i=1}^{r}g_{i}\cdot (Hv)}=\bigcup_{i=1}^{r} \overline{g_{i}\cdot (Hv)} =\bigcup_{i=1}^{r}g_{i}
\cdot \overline{Hv}.
\]
For the last
equation, note that each $g_{i}$ induces a homeomorphism of topological spaces
$V\rightarrow V$ with inverse map $g_{i}^{-1}$. Also note that in an arbitrary topological space, one has the general rule $\overline{A\cup
  B}=\overline{A}\cup\overline{B}$ for subsets $A$ and $B$, which justifies
the previous equation.  Now assume for a contradiction
$0\in \overline{Gv}$. Then $0\in g_{i}
\cdot \overline{Hv}$ for some $i$, hence $0=g_{i}^{-1}0\in \overline{Hv}$, a
contradiction. Therefore, $0\not\in \overline{Gv}$, and as $G$ is reductive
there is an $f\in\kk[V]^{G}_{+,\le\sigma(G,V)}$ with $f(v)\ne 0$. As $f$ is
clearly $H$-invariant, this shows that $\sigma(H,V)\le\sigma(G,V)$.
\end{proof}

Note that this lemma does not hold for arbitrary subgroups: Take the action of
the multiplicative group $\Gm=(\kk\setminus\{0\},\cdot)$ on $V=\kk$ by left
multiplication and $H$ the subgroup generated by a primitive $n$th
root of unity. Then $\kk[V]^{\Gm}=\kk$ and $\kk[V]^{H}=\kk[x^{n}]$, hence
$\sigma(\Gm,V)=0$ while $\sigma(H,V)=n$.

\begin{Prop}\label{BoundForSubgroups}
Let $G$ be a linear algebraic group, $H\subseteq G$ a closed subgroup of finite index
and $V$ an $H$-module. Then
\[
\sigma(H,V)\le\sigma(G,\Ind_{H}^{G}(V))\le\sigma(G),
\]
so in particular we have $\sigma(H)\le \sigma(G)$.
\end{Prop}

\begin{proof}
There is a natural $H$-equivariant embedding $V\hookrightarrow \Ind_{H}^{G}(V)$,
which turns $V$ into an $H$-submodule of $\Ind_{H}^{G}(V)$.
The restriction map $\Phi: \kk[\Ind_{H}^{G}(V)]^{G}\rightarrow \kk[V]^{H}$,
$f\mapsto f|_{V}$ is well defined and decreases degrees. Let $S\subseteq
\kk[\Ind_{H}^{G}(V)]_{+}^{G}$ be a $\sigma$-set for $G$. We will show that
$\Phi(S)\subseteq \kk[V]_{+}^{H}$ is a $\sigma$-set for $H$, which proves the proposition.
Take $v\in V$ such that $\Phi(f)(v)=f(v)=0$, for all $f\in S$. Since $S$ is a
$\sigma$-set for $G$, this means
$f(v)=\Phi(f)(v)=0$ for all $f\in \kk[\Ind_{H}^{G}(V)]_{+}^{G}$. By the
proof of Schmid \cite[Proposition 5.1]{Schmid}, the map $\Phi$ is surjective, so we
have $f(v)=0$ for all $f\in \kk[V]_{+}^{H}$. Thus, $\Phi(S)\subseteq
\kk[V]_{+}^{H}$ is a $\sigma$-set for $H$.
\end{proof}

An immediate consequence of Propositions \ref{propsigmasubgroup} and \ref{BoundForSubgroups} is

\begin{Corollary}\label{sigmaGG0}
Let $G$ be a linear algebraic group, with $G^0$ the connected component of $G$ containing the identity. We have the inequalities
\[
\sigma(G)\le[G:G^{0}]\sigma(G^{0})\quad \textrm{and}\quad \sigma(G^{0})\le\sigma(G).
\]
In particular, $\sigma(G)$ and $\sigma(G^{0})$ are either both
finite or infinite.
\end{Corollary}

\begin{rem}\label{betaSepFactorGroup}
If $G$ is a linear algebraic group and $N$ a closed normal subgroup of $G$,
then we have
\[
\sigma(G/N)\le\sigma(G).
\]
This follows from the fact that every $G/N$ module can be turned into a
$G$-module via the canonical map $G\rightarrow G/N$.
\end{rem}

Propositions \ref{propsigmanormsubgroup} and \ref{propsigmasubgroup} are
proved in \cite{DomokosCziszter} under the assumption that $[G:N]$ is not
divisible by $p$. Our proofs are rather similar to
\cite[Theorem~2]{KohlsKraft}. The proof of Proposition \ref{BoundForSubgroups}
is similar to \cite[Corollary~1]{KohlsKraft}.

\section{The $\sigma$-number for finite groups}\label{SecSigmaFinite}

We now specialize to the case of finite groups.  Throughout this section we work over an  algebraically closed field $\kk$ of characteristic $p$, which is assumed to divide $|G|$. Our first result is a generalisation of \cite[Corollary~5.3]{DomokosCziszter} to the modular case.

\begin{Theorem} Let $G$ be a group of the form $P \times A$, where $P$ is a $p$-group and $A$ is an abelian group of order not divisible by $p$. Then $\sigma(G) = |P|\exp(A)$.
\end{Theorem}

\begin{proof} 
We have $\sigma(G) \leq [G:A]\sigma(A)=|P|\sigma(A)$ by Proposition
  \ref{propsigmasubgroup}. By
  \cite[Corollary~5.3]{DomokosCziszter},  $\sigma(A) = \exp(A)$ when $|A|$ is not divisible by $p$, so we have proved   $\sigma(G) \leq |P|\exp(A)$. It remains to show that   $\sigma(G) \geq |P|\exp(A)$.
To this end, let $W$ be a 1-dimensional $\kk A$-module with character of order
$e:= \exp(A)$, and consider the  $P\times A$-module $V:= \kk P \otimes_{\kk}
W$, where $P$ acts on only the first factor and $A$ on only the second. We write
$\{v_g\mid g \in P\}$ for a basis of $V$ on which $P$ acts via the regular representation, with $\{x_g\mid g \in
P\}$ the dual basis. Notice that a homogeneous $f \in \kk[V]$ is $A$-invariant if and only
if $\deg(f)$ is divisible by $e$. Now consider the point $v:=\sum_{g \in P}v_g
\in V$. Take a homogeneous $f\in\kk[V]_{+}^{G}$ such that
$f(v)\ne 0$. As all monomials are eigenvectors under the $A$-action, every
monomial appearing in $f$ is $A$-invariant. As $\kk[V]^{P}$ is linearly
spanned by orbit sums of monomials $o_{P}(m)$, $f(v)\ne 0$ implies there
exists an $A$--invariant monomial $m$ of the same degree as $f$ such that $o_{P}(m)(v)\ne 0$.   
The proof of Proposition \ref{propdeltavreg} implies that $m=(\prod_{g \in
  P}x_g)^{d}$ for some $d\in\mathbb{N}$. In order for $m$ to be $A$-invariant, it
follows $e|\deg(m)=d|P|$. Since $|P|$ and $e$ are coprime, this means that
$e|d$, i.e.  $\deg(f) \geq e|P|$, and so   $\sigma(G,V) \geq e|P|=\exp(A)|P|$ as required.   
\end{proof}

From this result we obtain immediately:

\begin{Corollary}\label{SigmaCyclicGroups}
For any cyclic group $G$, we have $\sigma(G)=|G|$.
\end{Corollary}

The following is part (b) of Theorem \ref{thmsigmaleg}.

\begin{Theorem} \label{ThmCyclicFactorOfSylow}
Let $G$ be a finite group. Assume $P$
  is a Sylow-$p$-subgroup of $G$ such that $N_G(P)/P$ is cyclic. Then $\sigma(G) \geq|N_G(P)|$.
\end{Theorem}

\begin{proof} Firstly, $\sigma(G) \geq \sigma(N_G(P))$ by Proposition
  \ref{BoundForSubgroups}, so we may assume $G=N_G(P)$. It is enough to find a
  $G$-module $V$ with $\sigma(G,V) \geq |G|$. Set $r:= |G/P|$ and let $\zeta\in\kk$ be a
  primitive $r$th root of unity. By the Schur-Zassenhaus Lemma (see \cite[Theorem
  7.41]{RotmanGroups}), $P$ has a complement $H$ in $G$. Let $t$ be a generator
  of $H$.
Define a $\kk G$-module as follows: a $\kk$-basis is given by $\{v_g\mid g \in
P\}$, with dual basis $\{x_g\mid g \in P\}$. The action of $P$ on $V$ is via
the regular representation, while the action of $H$ is given via $t^{i}\cdot
v_{g}:=\zeta^{-i}v_{t^{i}gt^{-i}}$ for any $i\in\Z$ and $g\in P$.

Let $v:= \sum_{g \in P} v_g$, and let $f \in \kk[V]_{+}^G$ be homogeneous such
that $f(v)\ne 0$. Once more, $\kk[V]^{P}$ is linearly spanned by orbit sums of
monomials $o_{P}(m)$. As $r=|H|$ is invertible in $\kk$, $\kk[V]^{G}$ is linearly
spanned by elements $s_{m}:=\sum_{i=0}^{r-1}(t^{i}\cdot o_{P}(m))$. Therefore,
there exists a monomial $m$ of the same degree as $f$ such that
\[
0\ne s_{m}(v)=\sum_{i=0}^{r-1}(t^{i}\cdot o_{P}(m))(v)=\sum_{i=0}^{r-1}\zeta^{\deg(m)i}|Pm|.
\]
If the rightmost sum is to be non-zero, we must have again $|Pm|=1$,
that is, $m$ must be $P$-invariant, i.e. of the form $(\prod_{g\in P}x_{g})^{d}$ for
some $d\ge 1$. It follows
\[
0\ne s_{m}(v)=\sum_{i=0}^{r-1}\zeta^{d|P|i}.
\]
However, the sum on the right hand side is non-zero if and only if
$r|d$. Therefore, $m$ (hence $f$) has to be of degree at least $|P|r=|G|$. This shows that $\sigma(G,V) \geq |G|$ as required.
\end{proof}

\begin{Corollary}
Let $G$ be a finite group. Assume $P$
  is a Sylow-$p$-subgroup of $G$ and $\sigma(G)=|P|$. Then $N_{G}(P)=|P|$.
\end{Corollary}

\begin{proof}
Assume for a contradiction $N_{G}(P)\supsetneq P$ and take a $g\in
N_{G}(P)\setminus P$. Then the subgroup $H:=\langle P,g\rangle$ of $G$ satisfies
$N_{H}(P)=H$ and $H/P$ is cyclic. Hence by Propostion \ref{BoundForSubgroups}
 and Theorem~\ref{ThmCyclicFactorOfSylow} we would have
\[
\sigma(G)\ge \sigma(H)=|H|>|P|,
\]
 a contradiction.
\end{proof}

We will write $Z_{n}$ for a cyclic group of order $n$, which if convenient we
identify with $\Z/n\Z$. Recall $\Aut(Z_{n})\cong(\Z/n\Z)^{\times}$, which
is cyclic when $n$ is prime.

\begin{Prop}\label{PropZqZd}
Assume $p,q$ are primes and $d\in\N$ such that $p|d$ and $d|q-1$. Take
an embedding $Z_{d}\hookrightarrow \Aut(Z_{q})$ and form the corresponding
semidirect product $Z_{q}\rtimes Z_{d}$. Then over a field of characteristic
$p$, we have $\sigma(Z_{q}\rtimes Z_{d})=q$.
\end{Prop}

Note that over a field of characteristic $q$, $\sigma(Z_{q}\rtimes Z_{d})=dq$
by Theorem \ref{ThmCyclicFactorOfSylow}, and in the non-modular case,
$\sigma(Z_{q}\rtimes Z_{d})=q$ by Cziszter and Domokos \cite[Proposition
6.2]{DomokosCziszter}. This proof here is an adapted version to the modular case of the latter
proof by Cziszter and Domokos. We want to thank Cziszter for explaining some
details of their exposition to us via eMail. In the proof we will use a
decomposition of the regular representation of $G=Z_{q}\rtimes Z_{d}$ into a
direct sum of (not necessarily indecomposable) smaller modules which we construct below. 
We write  $G=\langle g,h\rangle$ such that 
\[
\ord(g)=d,\quad\quad\quad \ord(h)=q,
\]
and set $H:=\langle h\rangle\cong Z_{q}$ and $D:=\langle g\rangle\cong Z_{d}$.
Then with $k+q\Z$ a suitable element of multiplicative order $d$ in
$\Z/q\Z$  we have
\[
g^{a}h^{b}=h^{k^{a}b}g^{a} \quad\quad\text{ for all }a,b\in\Z.
\]
For convenience, we will  write $k^{-a}$ for a suitable
representative of the class $(k+q\Z)^{-a}$. Then $V_{\reg}$ has a basis $\{v_{g^{j}h^{r}}\mid
j=0,\ldots,d-1,\,\,r=0,\ldots,q-1\}$. We choose a primitive $q$th root of
unity $\zeta\in \kk$ and define
\[
w_{i,j}:=\sum_{r=0}^{q-1}\zeta^{-ir}v_{g^{j}h^{r}}\in V_{\reg}\quad\quad\text{ for
 } j=0,\ldots,d-1,\,\,\,i=0,\ldots,q-1.
\]

\begin{Lemma}\label{VregForZqZd}
For all $i=0,\ldots,q-1$, the vector space
\[
V_{i}:=\langle w_{i,0},w_{i,1},\ldots,w_{i,d-1}\rangle
\]
is a $G$-submodule of $V_{\reg}$, and we have a decomposition
\[
V_{\reg}=\bigoplus_{i=0}^{q-1}V_{i}.
\]
The action of $G$ on the summands is given by
\[
g^{a}\cdot w_{i,j} = w_{i,(j+a\mod d)} \quad \text{ and
}\quad h^{b}\cdot w_{i,j}=(\zeta^{i})^{k^{-j}b}w_{i,j}
\]
for $j=0,\ldots,d-1,\,\,\,i=0,\ldots,q-1$. 
\end{Lemma}

\begin{proof}
As $(\zeta^{-ir})_{i,r=0,\ldots,q-1}\in \kk^{q\times q}$ is a Vandermonde matrix
of full rank, we obtain for any $j=0,\ldots,d-1$ the equality of vector subspaces
\[
\langle w_{0,j},w_{1,j},\ldots,w_{q-1,j}\rangle=\langle v_{g^{j}h^{0}},v_{g^{j}h^{1}},\ldots,v_{g^{j}h^{q-1}}\rangle.
\]
Therefore, the set $\{w_{i,j}\}_{i,j}$ is a basis of $V_{\reg}$ and we get the
desired direct sum decomposition as vector spaces. We also see that
\begin{eqnarray*}
g^{a}\cdot w_{i,j}=g^{a}\cdot \sum_{r=0}^{q-1}\zeta^{-ir}v_{g^{j}h^{r}}
=\sum_{r=0}^{q-1}\zeta^{-ir}v_{g^{a+j}h^{r}}=w_{i,(j+a\mod d)}
\end{eqnarray*}
and
\begin{eqnarray*}
h^{b}\cdot w_{i,j}&=&h^{b}\cdot \sum_{r=0}^{q-1}\zeta^{-ir}v_{g^{j}h^{r}}
=\sum_{r=0}^{q-1}\zeta^{-ir}v_{h^{b}g^{j}h^{r}}
=\sum_{r=0}^{q-1}\zeta^{-i(r+k^{-j}b-k^{-j}b)}v_{g^{j}h^{k^{-j}b+r}}\\
&=&\zeta^{ik^{-j}b}\sum_{r=0}^{q-1}\zeta^{-i(r+k^{-j}b)}v_{g^{j}h^{r+k^{-j}b}}=\zeta^{ik^{-j}b}w_{i,j},
\end{eqnarray*}
as desired, and therefore the $V_{i}$'s are $G$-submodules.
\end{proof}

\begin{proof}[Proof of Proposition~\ref{PropZqZd}.]
As $Z_{q}$ is a subgroup of $G$, we have $q=\sigma(Z_{q})\le\sigma(G)$ by Corollary~\ref{SigmaCyclicGroups} and Proposition \ref{BoundForSubgroups}, so it remains
to show $\sigma(G)\le q$. By Corollary \ref{corsigmavreg}, Lemma~\ref{VregForZqZd} and Proposition~\ref{SigmaDirectSumFiniteGroups} we have
\[
\sigma(G)=\sigma(G,V_{\reg})=\sigma(G,\oplus_{i=0}^{q-1}V_{i})=\max\{\sigma(G,V_{i})\mid
i=0,\ldots,q-1\}.
\]
Note that $V_{0}$ is obtained from the regular representation of $Z_{d}$ and the
projection $G\rightarrow Z_{d}$. Therefore,
$\sigma(G,V_{0})\le\sigma(Z_{d})=d<q$. As the $\zeta^{i}$'s for $i=1,\ldots,q-1$ are
just different primitive roots of unity, the modules $V_{i}$ for $i=1,\ldots,q-1$ are
pairwise isomorphic, so it is enough to show $\sigma(G,V_{1})\le q$. We write
$V:=V_{1}$ and $K[V]=K[x_{0},\ldots,x_{d-1}]$.  The action on $K[V]$ then has
the following form:  For all $a,b\in\Z$ we have
\begin{eqnarray*}
g^{a} \cdot x_{j}&=&x_{(j+a\mod d)},\\
h^{b}\cdot x_{j}&=&\zeta^{-k^{-j}b}x_{j} \quad \text{ for all }j=0,\ldots,d-1.
\end{eqnarray*}
Note that  $k^{-j}$ is understood mod $q$ at all times. From this we see that a monomial
\[
x_{j_{1}}^{\alpha_{1}}\cdot\ldots \cdot x_{j_{s}}^{\alpha_{s}} \text{ is }
H\text{-invariant if and only if}\]\[ 
\alpha_{1}\bar{k}^{-j_{1}}+\ldots+ \alpha_{s}\bar{k}^{-j_{s}}=\bar{0}\in \Z/q\Z.
\]
Now for any non-empty subset $S:=\{j_{1},\ldots,j_{s}\}\subseteq \{0,\ldots,d-1\}$, we
consider the subset (of same length $s$)
$\{\bar{k}^{-j_{1}},\ldots,\bar{k}^{-j_{s}}\}\subseteq (\Z/q\Z)^{\times}$. By
\cite[Lemma 6.1]{DomokosCziszter} there exist
$\alpha_{1},\ldots,\alpha_{s}>0$ such that $\alpha_{1}+\ldots+\alpha_{s}\le
q$ and
\[
\alpha_{1}\bar{k}^{-j_{1}}+\ldots+ \alpha_{s}\bar{k}^{-j_{s}}=\bar{0}\in \Z/q\Z.
\]
We can thus define the monomial \[m_{S}:=x_{j_{1}}^{\alpha_{1}}\cdot\ldots
\cdot x_{j_{s}}^{\alpha_{s}}\in \kk[V]^{H},\] where $\alpha_1,\ldots, \alpha_s$ are chosen in such a way as to minimise $\alpha_{1}+\ldots+\alpha_{s}$.
We now claim that the common zero set of all the orbit sums
\[
o_{D}(m_{S}):=\sum_{m'\in D\cdot m_{S}}m'\in \kk[V]^{G}
\]
(for all non-empty subsets $S$) is $0$: otherwise, take $u=(u_{0},\ldots,u_{d-1})\in V\setminus\{0\}$  in the
common zero set of all those $o_{D}(m_{S})$. Consider the non-zero coordinates of $u$,
\[
S=\{j_{1},\ldots,j_{s}\}:=\{j\in\{0,\ldots,d-1\}\mid u_{j}\ne 0\}\ne \emptyset.
\]
By assumption, $o_{D}(m_{S})(u)=0$. We show this is a contradiction by considering
two cases. Define $m:=x_{j_{1}}\cdot\ldots\cdot x_{j_{s}}$ (which might be
different from $m_{S}$). 

\underline{First}, assume the
$D$-stabilizer of $m$ is trivial. Then every monomial in the orbit
$D\cdot m_{S}$ different from $m_{S}$ contains a variable outside
$\{x_{j_{1}},\ldots,x_{j_{s}}\}$, hence its value on $u$ is zero. Therefore,
\[
o_{D}(m_{S})(u)=m_{S}(u)=u_{j_{1}}^{\alpha_{1}}\cdot\ldots\cdot
u_{j_{s}}^{\alpha_{s}}\ne 0,
\]
a contradiction. 

\underline{Second}, assume the $D$-stabilizer of $m$ is non-trivial. 
So there exists a non-identity element $g^{a}\in D$ with $g^{a}\cdot m=m$. This
means
\[
\{j_{1}+d\Z,\ldots,j_{s}+d\Z\}=\{j_{1}+a+d\Z,\ldots,j_{s}+a+d\Z\},
\]
hence
\[
\underbrace{\bar{k}^{-j_{1}}+\ldots+\bar{k}^{-j_{s}}}_{=:w}=\bar{k}^{-j_{1}-a}+\ldots+\bar{k}^{-j_{s}-a}=\bar{k}^{-a}\cdot
w\in\Z/q\Z.
\]
Hence $(\bar{k}^{-a}-1)w=0$. As $\Z/q\Z$ is a field and $\bar{k}^{-a}\ne\bar{1}$, we get
\[
0=w=\bar{k}^{-j_{1}}+\ldots+\bar{k}^{-j_{s}},
\]
which means the monomial $m$ is $H$-invariant. By the minimality
assumption we have $m=m_{S}$. Now again, every monomial in the orbit
$D\cdot m$ different from $m$ contains a variable outside
$\{x_{j_{1}},\ldots,x_{j_{s}}\}$, hence its value on $u$ is zero. So we have 
\[
o_{D}(m)(u)=m(u)=u_{j_{1}}\cdot\ldots\cdot u_{j_{s}}\ne 0,
\]
a contradiction.
\end{proof}

In \cite[Theorem 7.1]{DomokosCziszter}, it is shown that for a
non-modular, non-cyclic group $G$, we have $\sigma(G)\le \frac{|G|}{l}$,
where  $l$ denotes the smallest prime
divisor of $|G|$. The following, which is part (c) of Theorem
\ref{thmsigmaleg}, is an extension of this to the modular case. 

\begin{Theorem}\label{propsigmapnilpotent}  
Let $G$ be a finite $p$-nilpotent group which has a non-normal Sylow-$p$-subgroup. If $l$ denotes the smallest prime
divisor of $|G|$, then
\[
\sigma(G)\le\frac{|G|}{l}.
\]
\end{Theorem}

\begin{proof}
Recall that a group $G$ is $p$-nilpotent if and only if it has a  Sylow-$p$-subgroup
$P$ of $G$ with a normal complement, i.e. a normal subgroup $H\unlhd G$ such
that $G=PH$ and $P\cap H$ is the trivial group. Let $l'$ denote the smallest prime divisor of
$|H|$. In case $H$ is not cyclic, by the aforementioned result of Cziszter and
Domokos, we have $\sigma(H)\le \frac{|H|}{l'}$, thus
$\sigma(G)\le\sigma(H)[G:H]\le\frac{|H|}{l'}[G:H]=\frac{|G|}{l'}\le\frac{|G|}{l}$. So
we may assume that $H\cong Z_{h}$ is cyclic of order $h$. We have a
group homomorphism 
\[
\varphi: P\rightarrow \Aut(H),\quad a\mapsto \varphi_{a}:
\left\{ 
\begin{array}{rcl}
H&\rightarrow& H\\
h &\mapsto & aha^{-1}.
\end{array}
\right.
\]
As by assumption $P$ is not a normal subgroup, we have
$\varphi(P)\ne\{\id_{H}\}$. Let $U=\ker(\varphi)$, and write
$\overline{\varphi}: P/U\rightarrow \Aut(H)$ for the induced injective morphism. Note
that $U\ne P$ as $\varphi(P)\ne\{\id_{H}\}$. We first show that $U$ is a
normal subgroup of $G$. By definition, $hu=uh$ for all $u\in U$ and $h\in
H$. As $U$ is a normal subgroup of $P$, for any $a\in P$ and $h\in H$ we
hence have $haU=hUa=Uha$, so indeed $U\unlhd G$. The canonical epimorphism
$P\twoheadrightarrow P/U$ induces an epimorphism
\[
G=HP\cong H\rtimes_{\varphi}P\twoheadrightarrow H\rtimes_{\overline{\varphi}}(P/U)
\]
with kernel $U$, hence we have $G/U\cong H\rtimes_{\overline{\varphi}}(P/U)$.
Let $l''$ denote the smallest prime-divisor of $G/U$. If we can show the claim
for $G/U$, i.e.
$\sigma(G/U)\le\frac{|G/U|}{l''}$, then
$\sigma(G)\le\sigma(G/U)\sigma(U)\le\frac{|G/U|}{l''}|U|=\frac{|G|}{l''}\le \frac{|G|}{l}$,
so we are done. Hence we can replace $G$ by $G/U$, i.e. we will assume that
$G\cong H\rtimes_{\varphi} P$ where $\varphi: P\hookrightarrow \Aut(H)$ is an
injective map and $P$ is a non-trivial $p$-group. We now choose a cyclic subgroup $Z_{p}$ of order $p$ of
$P$. The restriction of $\varphi$ to $Z_{p}$ is of course still injective. By
the same argument as before, it is enough to show the claim for the subgroup
$H\rtimes Z_{p}$ of $H\rtimes P$. Thus we now will assume that $G\cong
Z_{h}\rtimes_{\varphi} Z_{p}$ where $\varphi:
Z_{p}\hookrightarrow\Aut(Z_{h})\cong(\Z/h\Z)^{\times}$ is a monomorphism.
Therefore, the element $\varphi(1+p\Z)=a+h\Z$ is of multiplicative order $p$ in
$(\Z/h\Z)^{\times}$. We write $h=q_{1}^{s_{1}}\cdot\ldots\cdot q_{e}^{s_{e}}$
for the prime factorization of $h$ with different primes $q_{1},\ldots, q_{e}$.
The cyclic subgroups $U_{q_{i}}:=\langle \frac{h}{q_{i}}+ h\Z\rangle$ of $Z_{h}$ of order
$q_{i}$ are characteristic. Therefore for each $i$, we have an induced
homomorphism $\varphi_{i}: Z_{p}\rightarrow\Aut(U_{q_{i}})$. As $Z_{p}$ is of
prime order, this homomorphism
is either injective or trivial, where it is trivial if and only if
\[
a\cdot q_{1}^{s_{1}}\cdot\ldots\cdot q_{i}^{s_{i}-1}\cdot\ldots\cdot
q_{e}^{s_{e}}\equiv q_{1}^{s_{1}}\cdot\ldots\cdot q_{i}^{s_{i}-1}\cdot\ldots\cdot
q_{e}^{s_{e}}\mod q_{1}^{s_{1}}\cdot\ldots \cdot q_{e}^{s_{e}},
\]
i.e. if and only if $a\equiv 1 \mod q_{i}$. We want to show that at least one
of the maps $\varphi_{i}$ is injective. For the sake of a proof by
contradiction, we therefore assume $a\equiv 1\mod q_{i}$ for all $i=1,\ldots,e$. 
As $a$ has multiplicative order $p$ modulo $ h$, we have $a^{p}\equiv 1 \mod
q_{1}^{s_{1}}\cdot\ldots \cdot q_{e}^{s_{e}}$, so particularly $a^{p}\equiv
1\mod q_{i}^{s_{i}}$ for all $i=1,\ldots,e$. Lemma \ref{acong1apmod1} therefore
implies $a\equiv 1\mod q_{i}^{s_{i}}$ for $i=1,\ldots,e$, hence $a\equiv 1
\mod q_{1}^{s_{1}}\cdot\ldots \cdot q_{e}^{s_{e}}$, i.e. $a\equiv 1\mod h$, a
contradiction to $a+h\Z$ being of multiplicative order $p$. So we have that
$\varphi_{i}$ is injective for at least one $i$. Then for the subgroup
$U_{q_{i}}\rtimes_{\varphi_{i}}Z_{p}$ of $G=Z_{h}\rtimes_{\varphi} Z_{p}$, by
Proposition \ref{PropZqZd} we have
$\sigma(U_{q_{i}}\rtimes_{\varphi_{i}}Z_{p})=q_{i}=\frac{|U_{q_{i}}\rtimes_{\varphi_{i}}Z_{p}|}{p}$,
which as before implies $\sigma(G)\le\frac{|G|}{p}\le\frac{|G|}{l}$, so we are done.
\end{proof}

We have used the following number-theoretic lemma in the proof.

\begin{Lemma}\label{acong1apmod1}
Let $p,q>0$ be coprime, $s>0$ and $a\in\Z$. If
\[
a\equiv 1\mod q\quad\text{ and }\quad a^{p}\equiv 1\mod q^{s},
\]
then
\[
a\equiv 1\mod q^{s}.
\]
\end{Lemma}

\begin{proof}
We have $a=kq+1$ for some $k\in\Z$ by the first assumption. Hence by the second assumption,
\[
a^{p}=(1+kq)^{p}=1+kq \left(\sum_{i=1}^{p}{p\choose i}(kq)^{i-1}\right) \equiv 1\mod
q^{s}.
\]
Therefore,
\[
kq\left(\sum_{i=1}^{p}{p\choose
    i}(kq)^{i-1}\right)=kq\left(p+kq\left(\sum_{i=2}^{p}{p\choose
      i}(kq)^{i-2}\right)\right)\equiv 0\mod q^{s}.
\]
As $p,q$ are coprime, the second factor $p+kq(\cdots)$ is coprime to $q^{s}$,
and 
hence it follows $kq\equiv 0\mod q^{s}$. Thus we have $a=kq+1\equiv 1\mod q^{s}$, which is what we
wanted to prove.
\end{proof}

\begin{proof}[Proof of Theorem \ref{thmsigmaleg} (a).]
It remains only to show the second part of (a). If $\sigma(G)=|G|$, we have
already seen in Corollary \ref{corngpmodp} that $N_{G}(P)/P$ must be cyclic. Now assume
additionally $P$ is abelian and $G\ne P$.  If $N_{G}(P)=P$, then Burnside's Theorem (see
\cite[Theorem 7.50]{RotmanGroups}) implies $G$ is $p$-nilpotent, hence
$\sigma(G)<|G|$ by Theorem \ref{propsigmapnilpotent}, a
contradiction. Therefore, $N_{G}(P)/P$ must be non-trivial.
\end{proof}

It remains an open question to classify those finite groups which satisfy
$\sigma(G)=|G|$. Though we do not have any evidence, the following conjecture was a motivation
for many of our results:

\begin{Conj}
Suppose $G$ is a finite group. Let $P$ be a Sylow-$p$-subgroup of $G$. Then $\sigma(G)=|G|$ implies $P$ is normal in $G$.
\end{Conj}

Note that for $p$-nilpotent
groups, the conjecture follows from Theorem \ref{propsigmapnilpotent}.
From this conjecture, we would get the classification that  $\sigma(G)=|G|$ if and only if $P$ is normal in $G$ and $G/P$ is cyclic.  
Indeed, if $P$ is normal and $G/P$ is cyclic, $\sigma(G)=|G|$ by
Theorem \ref{ThmCyclicFactorOfSylow}. Conversely, if $\sigma(G)=|G|$ and the
conjecture holds, $P$ is
normal in $G$, and then as $|G|=\sigma(G)\le
|P|\sigma(G/P)$ (Proposition \ref{propsigmanormsubgroup}), the result of Cziszter and Domokos, Proposition
\ref{propdcnonmodularsigma}, forces $G/P$ to be cyclic. 

Also note that whenever $G$ contains a $p$-nilpotent subquotient with non-normal
Sylow-$p$-subgroup, $\sigma(G)<|G|$ by our relative results. So for the proof
of the conjecture, a classification of groups not containing such a
$p$-nilpotent subquotient could be the key.\\

In \cite[Question~1]{KohlsKraft}, the authors ask the similar (and also still
open) question which finite groups satisfy $\bsep = |G|$? At least, as a consequence of
Theorem \ref{ThmCyclicFactorOfSylow} and Proposition \ref{DeltaLeSigmaLeBeta},
we can add  groups $G$ with normal Sylow-$p$-subgroup $P$ and $G/P$ cyclic to the list.\\

We conclude with some explicit examples:

\begin{eg}
Assume throughout characteristic $2$. As $S_{3}\cong Z_{3}\rtimes Z_{2}$,
$\sigma(S_{3})=3$ by Proposition \ref{PropZqZd}. More generally, for
$D_{2q}$, the dihedral group of order $2q$ and $q$ an odd prime, we have
$\sigma(D_{2q})=\sigma(Z_{q}\rtimes Z_{2})=q$ by that proposition. Also note
that $\beta_{\sep}(D_{2q})=q+1$ by \cite[Theorem 8]{KohlsSezerD2p}. So here we
have the strict inequalities
$\delta(D_{2q})=2<\sigma(D_{2q})=q<\beta_{\sep}(D_{2q})=q+1$. 
The group $A_{4}$ has the normal
Sylow-$p$-subgroup $P=\langle (12)(34),(14)(23)\rangle$ of order $4$, and its
factor group $A_{4}/P$ is cyclic of order $3$. Hence $\sigma(A_{4})=12$ by
Theorem \ref{ThmCyclicFactorOfSylow}. As
$A_{4}\le S_{4}$, we have $12=\sigma(A_{4})\le \sigma(S_{4})$ by Proposition
\ref{BoundForSubgroups}. Also as $S_{3}\le S_{4}$ we have
$\sigma(S_{4})\le\sigma(S_{3})[S_{4}:S_{3}]=3\cdot 4=12$ by Proposition
\ref{propsigmasubgroup}. This shows $\sigma(S_{4})=12$. 
\end{eg}

\begin{eg}
Assume throughout characteristic $3$. Then $\sigma(S_{3})=6$ by Theorem~\ref{ThmCyclicFactorOfSylow}.
Furthermore $\sigma(A_{4})=4$: The projective indecomposable representations of $A_{4}$ are
obtained by induction of the characters of the Klein four group $H=\langle
(12)(34),(13)(24)\rangle\le A_{4}$, which leads to a three-dimensional
representation. Its matrix group is either the regular representation of
$Z_{3}$, or conjugate to
\[
G=\tiny\langle
\left(\begin{array}{ccc}
0&0&1\\
1&0&0\\
0&1&0
\end{array}\right),\left(\begin{array}{ccc}
2&0&0\\
0&2&0\\
0&0&1
\end{array}\right)\rangle.
\]
Then a computation with \Magma \cite{Magma} yields that the corresponding invariant
ring is minimally generated by 
\[
    x_1^2 + x_2^2 + x_3^2,\quad
    x_1 x_2 x_3,\quad 
    x_1^4 + x_2^4 + x_3^4,\quad 
    x_1^4 x_2^2 + x_1^2 x_3^4 + x_2^4 x_3^2.
\]
It is easily seen that the first three invariants in that list minimally cut
out $0$, which shows the claim.
\end{eg}

It is also worth mentioning that in the non-modular characteristics
(i.e. $p\ne 2,3$), $\sigma(A_{4})=4$ by
\cite[Corollary~4.2]{CziszsterThesis}.

\begin{eg}
Assume throughout arbitrary characteristic $p>0$. If $G$ is a $p$-group, we
have $\sigma(G)=|G|$ either from Theorem~\ref{ThmCyclicFactorOfSylow} or
Proposition~\ref{DeltaLeSigmaLeBeta} and Theorem~\ref{thmdeltag}.
For the dihedral groups
$D_{2p^{n}}$ with $n\ge 0$, we have $\sigma(D_{2p^{n}})=2p^{n}$, since in
characteristic~$2$ it would be a $p$-group, and by Theorem
\ref{ThmCyclicFactorOfSylow} otherwise. This strengthens the corresponding result
on $\beta_{\sep}$, see \cite[Proposition 10]{KohlsKraft}.
\end{eg}

\begin{ack}
This paper was prepared during visits of the first author to TU M\"unchen and
the second author to University of Aberdeen. The second of these visits was supported by the Edinburgh Mathematical Society's Research Support Fund. We want to thank Gregor Kemper
and the Edinburgh Mathematical Society for making these visits possible. We
also thank K. Cziszter for some helpful communication via eMail.
\end{ack}

\bibliographystyle{plain}
\bibliography{MyBib}

\end{document}